\documentclass[10pt,reqno,a4paper]{amsart}

\usepackage{
    amsmath,  amsfonts,  amssymb,   amsthm,    amscd,
    gensymb,   comment,   etoolbox,  url,
    booktabs, stackrel,  mathtools, enumitem,  mathdots,  microtype, lmodern,  mathrsfs,   tikz,  
    longtable, tabularx, float,     tikz,      pst-node, tikz-cd,   multirow, tabularx,  amscd,     bm, 
    array,     makecell, diagbox,   booktabs,  ragged2e, caption, subcaption}
\usepackage{slashbox}
\usepackage{graphicx,wrapfig}
\graphicspath{{./figures/}}
\usepackage{xcolor}
\usepackage[utf8]{inputenc}
\usepackage{microtype, fullpage, wrapfig,textcomp,csquotes,fbb}
\usepackage[colorlinks=true, linkcolor=blue, citecolor=blue, urlcolor=blue, breaklinks=true]{hyperref}
\usepackage[capitalise,nameinlink]{cleveref}
\usepackage{todonotes}
\usetikzlibrary{positioning}
\usetikzlibrary{shapes,arrows.meta,calc}

\usetikzlibrary{arrows}

\newtheorem{theorem}{Theorem}[section]

\newtheorem{definition}[theorem]{Definition}
\newtheorem{example}[theorem]{Example}
\newtheorem{lemma}[theorem]{Lemma}

\newtheorem{remark}[theorem]{Remark}

\newcommand{\Lk}{\mathrm{lk}}
\newcommand{\Dl}{\mathrm{del}}

\newcommand{\I}{\mathcal{I}}

\newcolumntype{x}[1]{>{\centering\arraybackslash}p{#1}}

\newcommand{\del}[1]{\textup{del}{\left(#1\right)}}
\newcommand{\link}[1]{\textup{lk}{\left(#1\right)}}
\newcommand{\wpoint}[1]{\underset{#1}{\vee}}

\usepackage{import}
\usepackage{xifthen}
\usepackage{pdfpages}
\usepackage{transparent}

\usepackage{multirow}

\renewcommand{\star}{*}

\begin{document}
\title{Independence complexes of wedge of graphs}

\author[N. Daundkar]{Navnath Daundkar}
\address{Department of Mathematics, Indian Institute of Technology Bombay, India}
\email{navnathd@iitb.ac.in}
\author[S. Panja]{Saikat Panja}
\address{Department of Mathematics, Harish Chandra Research Institute Prayagraj, India}
\email{saikatpanja@hri.res.in, panjasaikat300@gmail.com}
\author[S. Prasad]{Sachchidanand Prasad}
\address{Department of Mathematics, International Centre for Theoretical Sciences Bangalore, India}
\email{sachchidanand.prasad@icts.res.in,
sachchidanand.prasad1729@gmail.com}

\thanks{The first author is supported by IIT Bombay postdoctoral fellowship, the second author is supported by HRI PDF-M fellowship and the third author is supported by ICTS postdoctoral fellowship} 

\begin{abstract}
    In this article, we introduce the notion of a wedge of graphs and provide detailed computations for the independence complex of a wedge of path and cycle graphs. In particular, we show that these complexes are either contractible or wedges of spheres. 
\end{abstract}
\maketitle

\section{Introduction}\label{sec:intro}
A \emph{graph} is an order pair $G=(V,E)$, where $V$ is called the set of vertices and $E$ is called the set of edges. The set $E$ is consist of $2$-element subsets of $V$. A subject topological combinatorics is consist of the study of homotopy invariants
of certain cell complexes constructed using the graph to obtain combinatorial information about the graph $G$. The first example of this is the Lov{\'a}sz \cite{Lovaz} celebrated proof of the Kneser conjecture. The neighborhood complex $\mathcal{N}(G)$ of a graph $G$ is an abstract simplicial complex whose simplices are subsets of the vertex set $V$ having common neighbors.
Lov{\'a}sz uses the connectivity of neighborhood complex $\mathcal{N}(G)$ to compute a lower bound on the chromatic number of the corresponding graph. 
Lov{\'a}sz stated the similar conjecture that the lower bound on the chromatic number of a graph $G$ can be given in terms of the connectivity of certain cell complexes associated with cycle graph and $G$. These complexes are known as \emph{hom complexes}. Babson and Kozlov \cite{Lovazconj} proved Lov{\'a}sz conjecture, where they relate these hom complexes to \emph{independence complexes}.

 An abstract simplcial complex consist of all independent subsets of $V$ is called the \emph{independence complex} of $G$. It is denoted by $\mathcal{I}(G)$.  
 In the last two decades, the general problem of determining all the possible homotopy type of independence complexes for various classes of graphs has received considerable attention. For example, Kozlov \cite{Koz99} determined the homotopy type of independence complexes for path and cycle graphs, Kawamura \cite{indchordal, indforest} for chordal graphs and forests, Bousquet-M\'{e}lou, Mireille and Linusson, Svante and Nevo, Eran \cite{indsquaregrid} for square-grid graphs, Engstr{\"o}m \cite{Engstrclawfree} investigated this problem for claw-free graphs, and Raun \cite{Braunindstbknes} for stable Kneser graphs  etc. 

The purpose of this paper is to introduce the notion of a wedge of graphs and compute the all possible homotopy type of independence complexes of wedges of some classes of graphs. In particular, we consider the class of wedges of path graphs and cycle graphs. 
We also describe the relationships between combinatorial and topological invariants associated with the wedge of graphs with the corresponding combinatorial and topological invariants of the components. We hope these computations will be helpful in other parts of mathematics.

The paper is organized as follows: We begin \Cref{sec:prelim} by defining some combinatorial and topological objects associated with a graph. Then we introduce the notion of a wedge of graphs. We end this section by recalling some results about the homotopy type of independence complexes of path graphs and cycle graphs. 

Finally in \Cref{sec:res}, we prove the main results of this paper. Here, we compute all possible homotopy types of the wedge of finitely many path graphs and wedge of finitely many different cycle graphs. Next, we consider the wedge of cycle and path graphs with respect to different base points and comupute all possible homotopy types of corresponding independence complexes.

\section{Preliminaries}\label{sec:prelim}
In this section, we will define the main object of our study.
A couple of examples will follow this. 
Also, we will mention the \emph{fold lemma}, which will be a main ingredient to some of the proofs. Furthermore, we will mention some of the previously proven results, which will be 
used to prove our main theorems. 

\begin{definition}
An abstract simplicial complex $K$ is a collection of subsets of $\{v_1,\dots, v_n\}$, such that 
\begin{enumerate}
\item $\{ v_i \} \in K$ for all $1\leq i\leq n$,
\item if $\sigma \in K$ and $\tau\subseteq \sigma$, then $\tau\in K$.
\end{enumerate}
\end{definition}
The elements of $K$ are called faces. The \emph{dimension} of a face $\sigma$ is defined as $|\sigma|-1$. 
In this paper, by simplicial complex we mean a geometric realization of an abstract simplicial complex. Without loss of generality we use simplicial complex for an abstract simplicial complex.

Now we define an important simplicial complex associated with a graph.
\begin{definition}[Independence complex]
For a finite simple graph $G$, with the vertex set $V$, the independence complex $\mathcal{I}(G)$ is the simplicial complex consisting of
all independent (i.e. no two vertices are adjacent) subsets of $V$ as its simplices.
\end{definition}
The independence complex of a graph $G$ is denoted by $\mathcal{I}(G)$.

\begin{example}
We now see some examples of the independence complex of graphs.
\begin{enumerate}
\item Let $P_3$ be a path graph on theree vertices $\{1,2,3\}$. Then  
\[\I(P_3)=\bigg\{\{1\},\{2\},\{3\},\{1,3\}\bigg\}.\]
Hence the independence complex is homotopy equivalent to $\mathbb{S}^0$, which will be written as $\mathcal{I}(P_3)\cong \mathbb{S}^0$ (see \Cref{fig:indp3}).
\begin{figure}[H]
    \centering
    \includegraphics[scale=0.9]{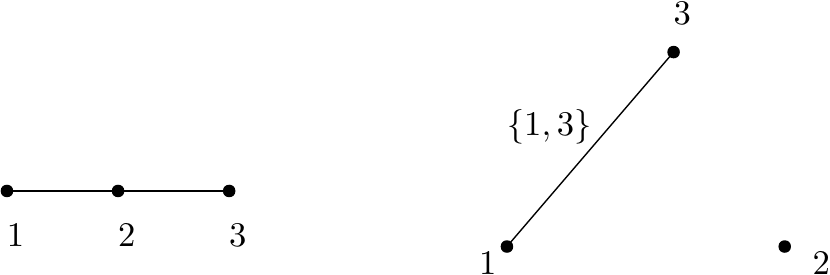}
    \caption{Path graph $P_3$ and $\I(P_3)$.}
    \label{fig:indp3}
\end{figure}

\item Let $C_4$ the cycle graph on $4$ vertices $\{1,2,3,4\}$. Then the independence complex is 
\[\I(C_4)=\bigg\{\{1\},\{2\},\{3\},\{4\},\{1,3\}, \{2,4\}\bigg\}.\]
Hence we get that $\mathcal{I}(C_4)\cong \mathbb{S}^0$ (see \Cref{fig:indc4}).
\begin{figure}[H]
    \centering
    \includegraphics[scale=0.9]{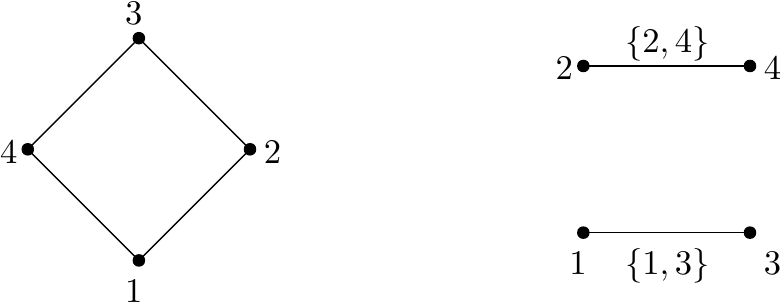}
    \caption{Cycle graph $C_4$ and $\I(C_4)$.}
    \label{fig:indc4}
\end{figure}
\end{enumerate}

\end{example}

A \emph{subcomlex} of a simplicial complex is a simplicial complex whose faces are contained in $K$.
There are two important subcomplexes associated with any simplicial complex.

\begin{definition}
Let $K$ be a (abstract) simplicial complex. The link of a vertex $v\in K$ is defined as 
\[\Lk (v,K) :=\{\sigma\in K \mid v\notin \sigma \text{ and } \sigma\cup {v}\in K \}.\]    
\end{definition}

\begin{definition}
Let $K$ be a (abstract) simplicial complex. The deletion of a vertex $v\in K$ is defined as 
\[\Dl(v,K) :=\{\sigma\in K \mid v\notin \sigma \} .\]    
\end{definition}
Observe that for any vertex in $K$, the subcomplex $\Lk(v,K)$ is a subcomplex of $\Dl(v,K)$.
The following is an important result which describes the homotopy type of a simplicial complex in terms of the link and deletion.
\begin{lemma}[{\cite[Lemma 2]{singhgrid}\label{lem: del link}}]
Let $K$ be a simplicial complex and $v\in K$ be a vertex such that $\Lk(v,K)$ is contractible in $\Dl(v,K)$. Then $K\simeq \Dl(v,K)\vee \sum \Lk(v,K)$.   
\end{lemma}
Now we define the central object of study in this paper. This will be followed by a few examples. Note that the definition depends on the choice of the so-called `wedge point'.
\begin{definition}[Wedge of graphs]
Given a finite family of graphs $G_i=(V_i,E_i)$, and $a_i\in V_i$ for all $i$, an \emph{wedge of graphs} is defined to be a graph $G=(V,E)$ such that 
\begin{align*}
        V&=\left(\bigcup_i V_i\setminus\{a_i\}\right)\cup \{a\},\\
        E&=\left(\bigcup_iE_i\setminus\{e_j\in E_i|a_i\in\partial(e_j)\}\right)\cup\left\{ab_k|a_ib_k\in E_i\right\}.
\end{align*}
    The common point $a$ will be called \emph{wedge point}.
\end{definition}

\begin{remark}
Let $\chi(G)$ be the chromatic number of a graph. Then we can observe that \[\chi(G_1\vee G_2)=\text{max}\{\chi(G_1),\chi(G_2)\}.\]    
\end{remark}

\begin{example}
    Consider two path graphs $P_3$ and $P_4$  on $3,4$ vertices respectively. Then choosing different wedge points, we can obtain different wedge graphs. We describe some of them here (see \Cref{fig:p3p4wedge}).

    \begin{figure}[H]
        \centering
        \includegraphics[scale=.7]{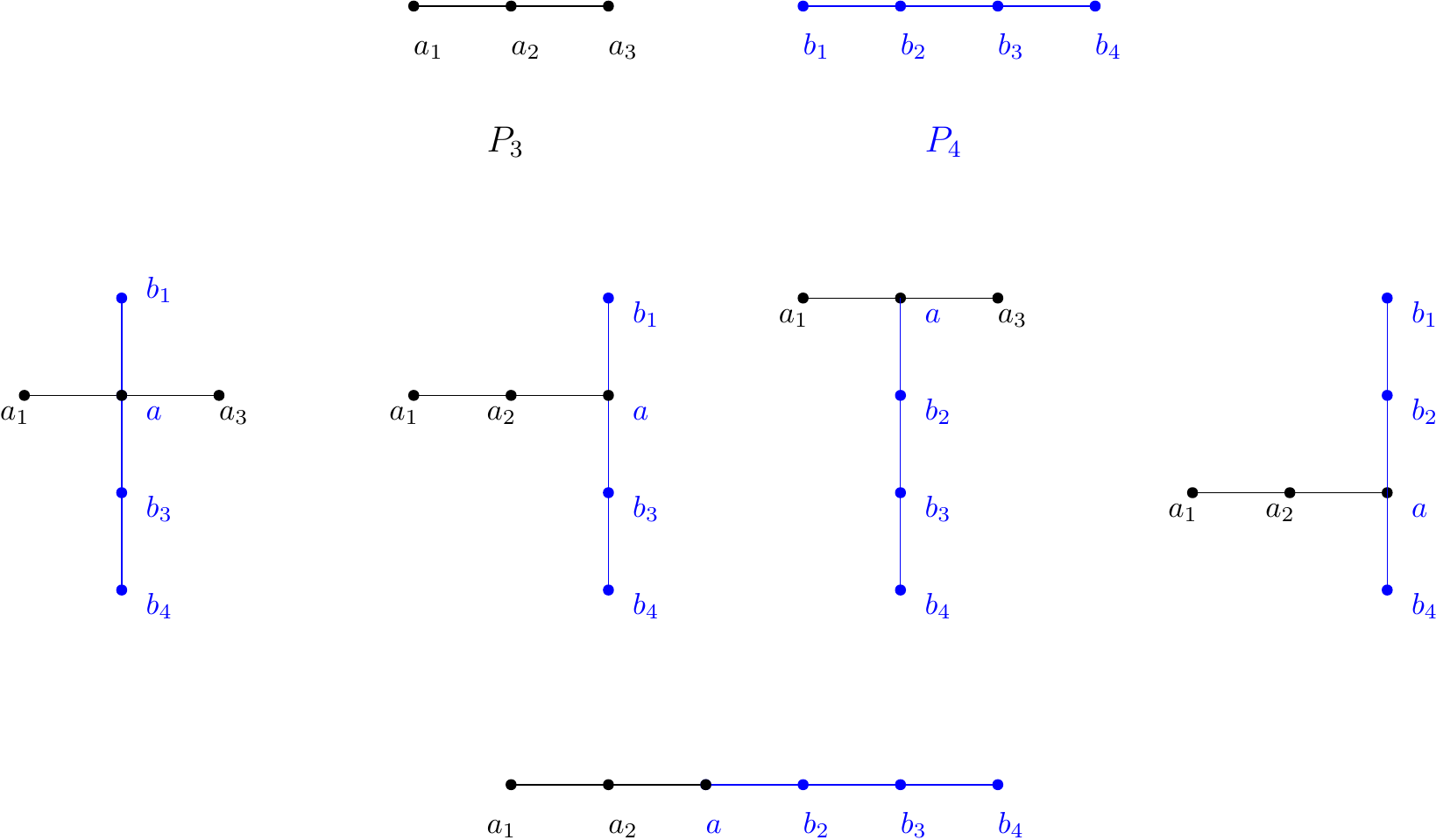}
        \caption{Different wedges graphs of $P_3$ and $P_4$}
        \label{fig:p3p4wedge}
    \end{figure}
\end{example}
\begin{remark}
Note that if we vary the wedge points, the obtained wedge graphs \textcolor{black}{need not be} isomorphic. Furthermore, they do not have a homotopic independence complex as well.
Consider the first and the last wedge graph of $P_3$ and $P_4$ described in \Cref{fig:p3p4wedge}. Let us denote them by $G$ and $H$, respectively. Then one can see that $G$ and $H$ are not isomorphic. Moreover,  $\I(G)\simeq \star$ and $\I(H)\simeq S^1$.
\end{remark}
Note that we have introduced the main two terms of the title, we mention one of the key ingredients for the proof of our theorems. \begin{theorem}[Fold lemma]\label{lemma:foldLemma}\cite[Lemma 3.4]{Engstrclawfree}
    Let $G$ be a graph, and $v\neq w$ vertices of $G$. If $N(v)\subseteq N(w)$ then the inclusion $\mathcal{I}(G\setminus w )\hookrightarrow \mathcal{I}(G)$ is a homotopy equivalence.  
\end{theorem}

We will end this section by mentioning two important results, due to Kozlov. This will be required in the proof of our theorem, as we are going to consider a few classes of wedges of these graphs. 

\begin{theorem}[{\cite[Proposition 4.6]{Koz99}\label{thm:ICofPm}}]
    Let $P_m$ be the path graph on $m$ vertices. Then
    \begin{equation}\label{eq:ICofPm}
        \mathcal{I}(P_m) \simeq 
        \begin{cases}
            \mathbb{S}^{k-1},& \text{ if } m=3k \\
            \text{pt},       & \text{ if } m=3k+1 \\
            \mathbb{S}^{k},  & \text{ if } m=3k+2.
        \end{cases}
    \end{equation}
\end{theorem}

\begin{theorem}[{\cite[Proposition 5.2]{Koz99}\label{thm:ICofCn}}]
    Let $C_n$ be the cycle graph on $n$ vertices. Then
    \begin{equation}\label{eq:ICofCn}
        \mathcal{I}(C_n) \simeq 
        \begin{cases}
            \mathbb{S}^{k-1}\vee \mathbb{S}^{k-1},& \text{ if } n=3k \\
            \mathbb{S}^{k-1},       & \text{ if } n=3k+1 \\
            \mathbb{S}^{k},  & \text{ if } n=3k+2.
        \end{cases}
    \end{equation}
\end{theorem}

\section{Main results}\label{sec:res}
We begin this section by proving a sufficient condition for the link of a vertex to be contractible inside deletion. This is crucial in using the \Cref{lem: del link}. Then we compute the homotopy type of wedge of path graphs and a wedge of cycle graphs.  
\begin{lemma}\label{lemma:contractibleReduction}
    Let $G$ be a graph, and $v\in \mathcal{I}(G)$ be a vertex of $G$. Let $\sigma\in \del{v}$ be a maximal simplex such that $\sigma\notin \link{v,\I(G)}$ and $\del{v,\I(G)}\setminus \{\sigma\}$ is contractible. Then
    \begin{displaymath}
        \mathcal{I}(G) \simeq \del{v,\I(G)}\vee \Sigma \left( \link{v,\I(G)}\right).
    \end{displaymath}   
\end{lemma}
\begin{proof}
    Since $\sigma\notin \link{v},~\link{v}\subseteq \del{v}\setminus\{\sigma\}$, and $\del{v}\setminus\{\sigma\}$ is contractible hence we have $\link{v}$ is contractible in $\del{v}$. Then the result follows from \Cref{lem: del link}.
\end{proof}
We start with the case when the wedge is taken to be of two path graphs. We give a complete description of all the cases.
In the later part of the section, this will be generalized in the case of a wedge of finitely many path graphs, using the concept of the terminal wedge.
\begin{theorem}\label{thm:pathWedgePath}
    Let $P_l$ be the path graph on $l$ vertices. Then 
    \begin{equation}\label{eq:pathWedgePath}
        \mathcal{I}\left(P_m \wpoint{a} P_n \right)
    \end{equation}
    is either a point or a sphere.
\end{theorem}
\begin{proof}
    If the wedge point is the terminal point for both of the graphs, then the wedge graph is $P_{m+n-1}$ hence the independence complex can be found using \autoref{thm:ICofPm}. Now suppose that $a=a_m=b_l,~m\ge 4$. Then Using \autoref{lemma:foldLemma} (with $v=a_1$ and $w=a_3$), we have that 
    \begin{align*}
        \mathcal{I}(P_m\wpoint{a}P_n) \simeq \mathbb{S}^0 \star \mathcal{I}(P_{m-3} \wpoint{a} P_n).
    \end{align*} 
    Thus we only need to compute the independence complexes for the cases $m=1,2$ and $3$ (Referring to the diagram \Cref{fig:PmWedgePn_type-1}, we assume $n=l+k$).
    \begin{figure}[H]
        \centering
    \def\svgwidth{0.5\columnwidth}
    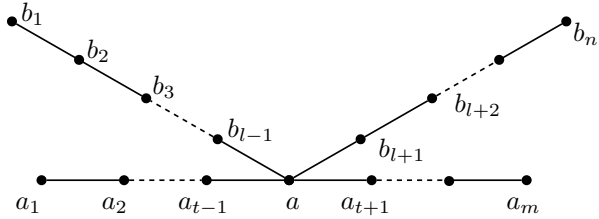

        \caption{Wedge of path graphs.}
        \label{fig:PmWedgePn_type-1}
    \end{figure}\
    \paragraph{\textcolor{black}{\textbf{Case 1} $m=1$}} In this case the wedge is same as $P_n$.

    \paragraph{\textcolor{black}{\textbf{Case 2} $m=2$}} Take $v=a_1$ and $w=b_{l-1}$. Then using \autoref{lemma:foldLemma}
    \begin{displaymath}
        \mathcal{I}(P_2\wpoint{a}P_n) \simeq \mathcal{I}(P_{l-2}) \star \mathcal{I}(P_{n-l+2}).
    \end{displaymath}  

    \paragraph{\textcolor{black}{\textbf{Case 3} $m=3$}} Take $v=a_1$ and $w=a$. Then using \autoref{lemma:foldLemma},
    \begin{displaymath}
        \mathcal{I}(P_3\wpoint{a}P_n) \simeq \mathbb{S}^0 \star \mathcal{I}(P_{l-1}) \star \mathcal{I}(P_{n-l}).
    \end{displaymath}
    Now, the result follows from \cref{eq:ICofPm}.
\end{proof}
The case of a wedge of the cycle graphs is not straightforward and hence we need to do a more careful analysis. This is an almost replica of Kozlov's work. 
But here we have more than one cycle, so we need to consider the maximal simplex appropriately inside the deletion which doesn't belong to the link. Here is the statement of the main result. 
\begin{lemma}\label{lemma: Gmn as delvslk}
 Let $C_m \wpoint{a} C_n$ be the wedge of cycle graphs and $G_{m,n}=\mathcal{I}\left(C_m \wpoint{a} C_n \right)$. Then \[G_{m,n}\simeq \del{a, G_{m,n})}\vee \sum \link{a,G_{m,n}}.\]
\end{lemma}
\begin{proof}
We denote $\I\left(C_m \wpoint{a} C_n \right)$ by $G_{m,n}$.
 We show that    $\link{a,G_{m,n}}$ is contractible in $\del{a, G_{m,n}}$. 
 Then the result follows from \Cref{lem: del link}. 
 Observe that, in general we have 
\[\link{a,G_{m,n}}\simeq \I(P_{m-2})*\I(P_{n-2}) ~~\text{ and }~~ \del{a,G_{m,n}}\simeq \I(P_{m-1})*\I(P_{n-1}).\]
 We consider the following cases:
 
\noindent{}\textbf{Case 1}\emph{ $n=3k$ or $m=3k'$.}
\vspace{.2mm}

In this case, either $\I(P_{m-2})\sim \I(P_{3(k-1)+1})$ or $\I(P_{n-2})\simeq \I(P_{3(k'-1)+1})$. From \Cref{thm:ICofPm}, we get that either $\I(P_{m-2})$ or $\I(P_{n-2})$ is contractible. Therefore, $\link{a,G_{m,n}}$ is contractible. This, from \Cref{lem: del link} result follows. 

\noindent{}\textbf{Case 2}\emph{ $n=3k+1$ or $m=3k'+1$.}
\vspace{.2mm}

Now consider the maximal simplex  \[\sigma=\{a_1,a_3,\dots,a_{3k-3},a_{3k-1}\}\cup\{b_1,b_3,\dots,b_{3k'-3},b_{3k'-1}\}.\] Then $\sigma\in\del{a,G_{m,n}}\setminus \link{a,G_{m,n}}$. We also have $\del{a,G_{m,n}}\simeq \I(P_{3k})*\I(P_{3k'})\simeq \mathbb{S}^{k+k'+1}$ using \Cref{thm:ICofPm}. Therefore, $\link{a,G_{m,n}}\setminus {\sigma}$ is contractible. Now  the result follows from \Cref{lemma:contractibleReduction}. 

\noindent{}\textbf{Case 2}\emph{ $n=3k+1$ and $m=3k'+2$.}
\vspace{.2mm}

In this case, $\del{a,G_{m,n}}$ is contractible. Therefore, the result follows from \Cref{lem: del link}.

In the remaining cases 

\noindent{}\textbf{Case 3}\emph{ $n=3k+2$ and $m=3k'+1$}
 and 

\noindent{}\textbf{Case 4}\emph{ $n=3k+2$ and $m=3k'+2$}
\vspace{.2mm} we get that $\del{a,G_{m,n}}$  is contractible. Therefore, the result follows from \Cref{lem: del link}.
Finally, in any case we have $G_{m,n}\simeq \del{a, G_{m,n})}\vee \Sigma \link{a,G_{m,n}}$.
\end{proof}
Now we consider the general case, where we have a wedge of $k$-many cycle graphs, with the wedge point to be $a$. Consider the cycle graphs $C_{m_i}$ with vertices $\{a_1^i,a_2^i,\ldots,a_{m_i}^i\}$. Let $G=\vee_{i=1}^kC_{m_i}$. Then observe that \[\link {a,\I(\bigvee_{i=1}^k C_{m_i})}\simeq *_{i=1}^k\I(P_{m_i-2}) ~~\text{ and }~~ \del{v,\I(\bigvee_{i=1}^k C_{m_i})}\simeq *_{i=1}^k\I(P_{m_i-1}).\] Now one can see that if one of the $m_i$ is of the form $3k$ or $3k+2$, then either $\link{a,}$  or $\del{a,}$ is contractible. Now if all $m_i$'s are of the form $3l_{i}+1$, then consider the maximal simplex  \[\sigma=\{a^1_1,a^1_3,\dots,a^1_{3k-3},a^1_{3l_1-1}\}\cup\dots\cup\{a^k_1,a^k_3,\dots,a^k_{3k'-3},a^k_{3l_k-1}\}.\] 
Note that $\sigma\in \del{a, G_{m_1,\dots,m_k}}\setminus \link{a,G_{m_1,\dots,m_k}}$. Observe that $\del{a,G_{m_1,\dots,m_k}}\simeq \mathbb{S}^{\Sigma_{i}^{k}l_i-1}$. Therefore, $\del{a,G_{m_1,\dots,m_k}}\setminus\{\sigma\}$ is contractible.
Therefore, using \Cref{lemma:contractibleReduction} we have the following.

\begin{lemma}\label{lem:gen-cyc-wedge}
Let $G=\vee_{i=1}^{k}C_{m_i}$ be the wedge of $k$-many cycle graphs and $G_{m_1,\dots,m_k}$ be its independence complex. Then $G_{m_1,\dots,m_k}\simeq \del{a,G_{m_1,\dots,m_k}}\vee \Sigma \link{a,G_{m_1,\dots,m_k}}$.    
\end{lemma}
We are now ready to compute the homotopy type of the independence complex of a wedge of cycle graphs. Although this will be further generalized in the later part, we present the complete computation to ease the reader's mind.
\begin{theorem}\label{thm:cycleWedgeCycle}
Let $C_m \wpoint{a} C_n$ be the wedge of cycle graphs and $G_{m,n}$ be the independence complex of $C_m \wpoint{a} C_n$. Then $G_{m,n}$ is contractible or homotopy equivalent to  wedge of spheres.
\end{theorem}
\begin{proof}
        \label{tab:label}

We consider the following cases: 

\noindent{}\textbf{Case 1} \emph{either  $n=3k$ or $m=3k'$}
\vspace{.2mm}

In this case $\link{a,G_{m,n}}$ is contractible as we saw in the first case of \Cref{lemma: Gmn as delvslk}. Therefore, by \Cref{lemma: Gmn as delvslk}, we have $G_{m,n}\simeq \del{a,G_{m,n}}\simeq \I(P_{m-1})*\I(P_{n-1})$. Now we have following subcases using \Cref{thm:ICofPm}: 
\begin{enumerate}
\item \emph{If $m=3k$ and $n=3k$}. In this case, $G_{m,n}\simeq\del{a, G_{m,n}}\simeq \mathbb{S}^{k-1}*\mathbb{S}^{k'-1}=\mathbb{S}^{k+k'-1}$.
\item \emph{If $m=3k$ and $n=3k'+1$.} In this case, $G_{m,n}\simeq\del{a, G_{m,n}}\simeq \mathbb{S}^{k-1}*\mathbb{S}^{k'}=\mathbb{S}^{k+k'}$.
\item \emph{If $m=3k$ and $n=3k'+2$.} In this case, $G_{m,n}\simeq\del{a, G_{m,n}}\simeq \mathbb{S}^{k-1}* \{pt\}=\{pt\}$.
\item \emph{If $m=3k+1$ and $n=3k'$.} In this case, $G_{m,n}\simeq\del{a, G_{m,n}}\simeq \mathbb{S}^{k}*\mathbb{S}^{k'-1}=\mathbb{S}^{k+k'}$.
\item \emph{If $m=3k+2$ and $n=3k'$.} In this case, $G_{m,n}\simeq\del{a, G_{m,n}}\simeq \{pt\}* \mathbb{S}^{k'-1} =\{pt\}$.
\end{enumerate}

\noindent{}\textbf{Case 2} \emph{$m=3k+1$, $n=3k^\prime +1$.}
\vspace{.2mm}

Consider the following diagram \Cref{fig: Gmn} of wedge of two cycles where the wedge point $a$ is taken to be $a_{3k+1}=b_{3k^\prime +1}$. 

\begin{figure}[h]
    \centering
    \includegraphics[scale=1]{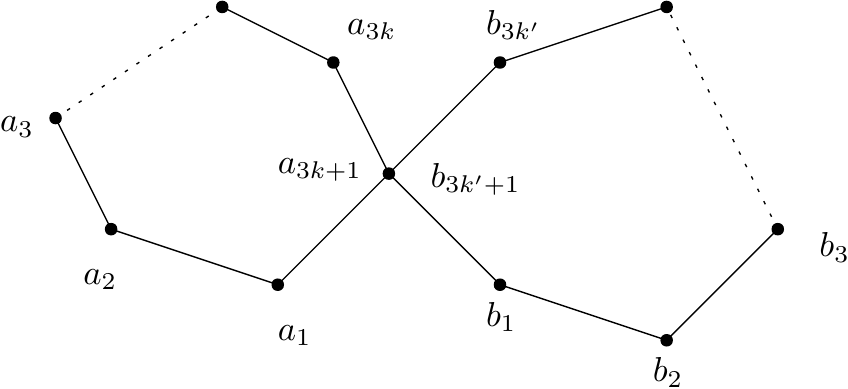}
    \caption{Wedge of $C_{3k+1}$ and $C_{3k'+1}$}
    \label{fig: Gmn}
\end{figure}

    \noindent Consider the simplex
    \[\sigma = \left\{ a_1,a_3,a_6,a_9,\cdots,a_{3k} \right\} \cup \left\{ b_1,b_3,b_6,\cdots,b_{3k^\prime } \right\} .
    \] 
    Since $\sigma \in \del{a}\setminus \link{a}$ and $\del{a}\setminus \{\sigma \}$ is contractible in $\del{a}$, we get using \Cref{lemma:contractibleReduction} that 
    \begin{align*}
        G_{m,n} & = \del{a,G_{m,n}}\vee \Sigma \link{a,G_{m,n}}\\
        & = \mathbb{S} ^{k+k^\prime -1} \vee \mathbb{S} ^{k+k^\prime},
    \end{align*}
    as $\del{a,G_{m,n}}\simeq \I(P_{m-1})*\I(P_{n-1})\simeq \I(P_{3k})*\I(P_{3k'})\simeq \mathbb{S}^{k+k'-1}$ and $\link{a, G_{m,n}}\simeq \I(P_{m-2})*\I(P_{n-2})\simeq \I(P_{3k-1})*\I(P_{3k'-1})\simeq \I(P_{3(k-1)+2})*\I(P_{3(k'-1)+2})\simeq \mathbb{S}^{k+k'-1}$.

\noindent{}\textbf{Case 3} \emph{$m=3k+2$, $n=3k^\prime +1$.}
\vspace{.2mm}  

In this case, we have $\del{a,G_{m,n}}$
is contractible as shown in \Cref{lemma: Gmn as delvslk}. Therefore, using \Cref{lemma: Gmn as delvslk}, we have $G_{m,n}\simeq \Sigma \link{a,G_{m,n}}$. Therefore,
\begin{align*}
    G_{m,n} & \simeq \Sigma \bigg(\I(P_{m-2})*\I(P_{n-2})\bigg)\\
    &\simeq \Sigma \bigg(\I(P_{3k})*\I(P_{3(k'-1)+2})\bigg)\\
    &\simeq \Sigma \bigg(\mathbb{S}^{k-1}*\mathbb{S}^{k'-1} \bigg) \simeq \mathbb{S}^{k+k'}.
\end{align*}

\noindent{}\textbf{Case 4} \emph{$m=3k+1$, $n=3k^\prime +2$.}
\vspace{.2mm}  

In this case, we have $\del{a,G_{m,n}}$
is contractible as shown in \Cref{lemma: Gmn as delvslk}. Therefore, using \Cref{lemma: Gmn as delvslk}, we have $G_{m,n}\simeq \Sigma \link{a,G_{m,n}}$. Therefore,
\begin{align*}
    G_{m,n} & \simeq \Sigma \bigg(\I(P_{m-2})*\I(P_{n-2})\bigg)\\
    &\simeq \Sigma \bigg(\I(P_{3(k-1)+2})*\I(P_{3k'})\bigg)\\
    &\simeq \Sigma \bigg(\mathbb{S}^{k-1}*\mathbb{S}^{k'-1} \bigg) \simeq \mathbb{S}^{k+k'}.
\end{align*}

\noindent{}\textbf{Case 5} \emph{$m=3k+2$, $n=3k^\prime +2$.}
\vspace{.2mm}  

In this case, we have $\del{a,G_{m,n}}$
is contractible as shown in \Cref{lemma: Gmn as delvslk}. Therefore, using \Cref{lemma: Gmn as delvslk}, we have $G_{m,n}\simeq \Sigma \link{a,G_{m,n}}$.
Therefore,
\begin{align*}
    G_{m,n} & \simeq \Sigma \bigg(\I(P_{m-2})*\I(P_{n-2})\bigg)\\
    &\simeq \Sigma \bigg(\I(P_{3k})*\I(P_{3k'})\bigg)\\
    &\simeq \Sigma \bigg(\mathbb{S}^{k-1}*\mathbb{S}^{k'-1} \bigg) \simeq \mathbb{S}^{k+k'}.
\end{align*}
This proves the theorem.
\end{proof}
Now we present the generalization of the previous theorem. We hope the path to generalization will be clear to the reader. We present the theorem without a detailed proof here, to avoid cumbersomeness.
\begin{theorem}
Let $\bigvee_{i=1}^k C_{m_i}$ be the wedge of $k$-many cycle graphs. Then $\I(\bigvee_{i=1}^k C_{m_i})$ is contractible or wedge of spheres.   
\end{theorem}
\begin{proof}
The proof follows from using  \Cref{lem:gen-cyc-wedge}  and induction on $k$.
\end{proof}

The next class of graphs is a wedge of a path and a cycle graph. Assume $C_n$ to be the cycle graph whose vertex set is $\{a_1,\dots, a_n\}$ oriented counterclockwise and $P_m$ be the path graph with vertex set is $\{b_1\dots, b_n\}$. For $1\leq k\leq \lceil n/2 \rceil$, consider the wedge graph $G_k=C_n\wpoint{c_k=a_1\sim b_k}P_m$. Note that the function $\lceil~\rceil$ appears to avoid obvious isomorphic classes of graphs. See \cref{fig:my_label} for a visual. Next we compute the independence complex of this graph.
\begin{figure}[h]
    \centering
    \includegraphics[scale=.9]{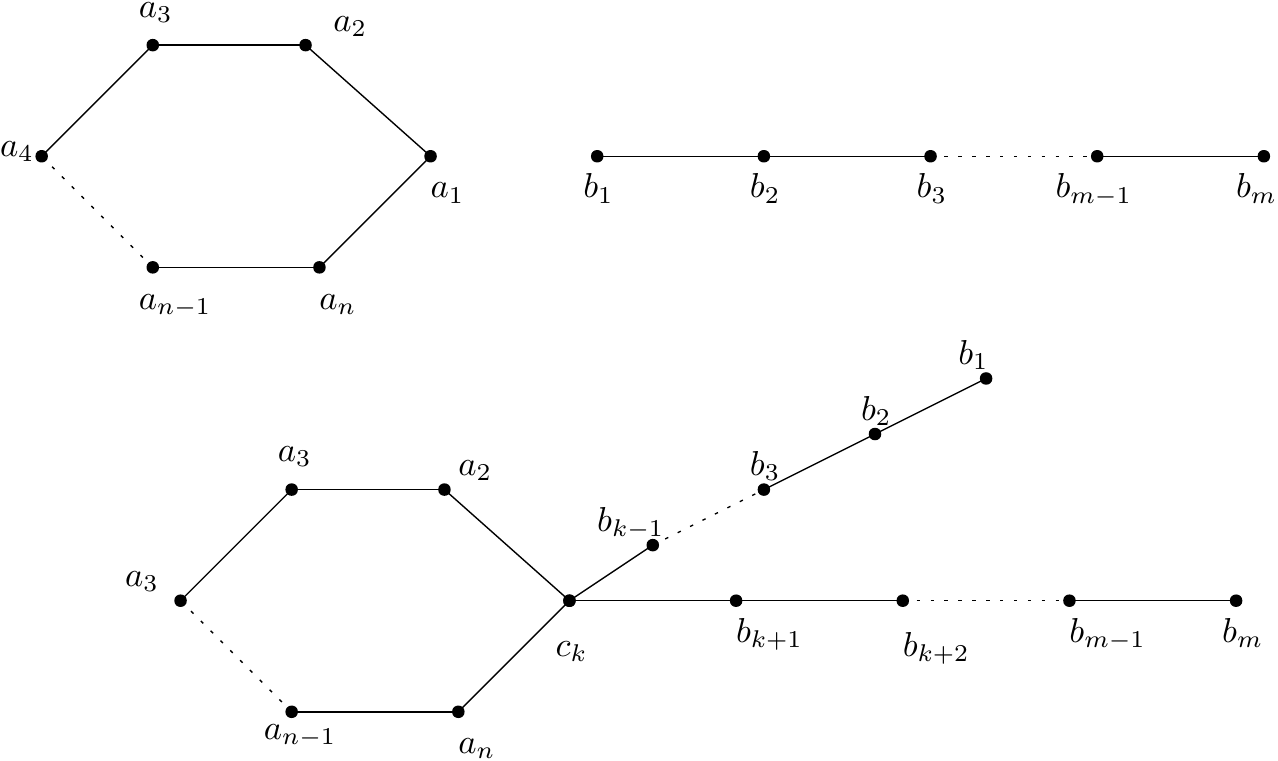}
    \caption{A graph $G_k$}
    \label{fig:my_label}
\end{figure}
\begin{theorem}\label{thm:pathWedgeCycle}
 The independence complex  $\mathcal{I}(G_k)$ is homotopy equivalent to either a point or wedge of spheres.
\end{theorem}
\begin{proof}
Let $c_k$ be the wedge point of the graph $G_k$. 
Let $C_n(n-3)$ be the subgraph of $G_k$ on vertices $\{a_3,\dots,a_{n-1}\}$. One can see that $C_n(n-3)$ is isomorphic to the path graph $P_{n-3}$ on $n-3$ vertices. Let $P_m(k-2)$ be the subgraph of $P_m$ on vertices $\{b_1,\dots,b_{k-2}\}$ and $P_m(m-k-1)$ be the subgraph of $P_m$ on $\{b_{k+2},\dots, b_{m}\}$. One can see that $P_{m}(k-2)$ is isomorphic to the path graph $P_{k-2}$ and $P_{m}(m-k-1)$ is isomorphic to $P_{m-k-1}$.
Note that the link $\Lk(c_k,\I(G_k))$ can be described as 
\[\bigg \{ A\sqcup B\sqcup C \mid A\in \I(C_n(n-3)), B\in P_m(k-2) \text{ and } C\in  P_{m}(m-k-1)  \bigg\}.\]

Therefore, we get the following \[\Lk(c_k,\I(G_k))\simeq \I(P_{n-3})*\I(P_{k-2})*\I(P_{m-k-1})\] and \[\Dl(c_k,\I(G_k))\simeq \I(P_{n-1})*\I(P_{k-1})*\I(P_{m-k}).\] It is easy to see that $\Lk(c_k,\I(G_k))$ is contractible in $\Dl(c_k,\I(G_k))$. Then by \Cref{lem: del link}, we get $\I(G_k)\simeq
\Dl(c_k,\I(G_k))\vee \Sigma \Lk(c_k,\I(G_k))$. \textcolor{black}{Let $\alpha =a+b+c$. }

\begin{table}[H]
    \centering
    \begin{tabular}{|c|c|c|c|c|c|c|c|}
        \hline
        & & \multicolumn{3}{c|}{$\Dl$} & \multicolumn{3}{c|}{$\Lk$} \\
        \hline
        & \backslashbox{$m-k$ }{$k$ } & $3b$ & $3b+1$ & $3b+2$ & $3b$ & $3b+1$ & $3b+2$  \\ 
        \hline 
        \multirow[c]{3}{*}{$n=3a$} & $3c$ & $\mathbb{S} ^{\alpha-1}$ & $\mathbb{S} ^{\alpha-1}$ & \text{pt} & \text{pt} & \text{pt} & \text{pt} \\ 
        \cline{2-8}
        & $3c+1$ & \text{pt} & \text{pt} & \text{pt} & \text{pt} & $\mathbb{S} ^{\alpha}$ & $\mathbb{S} ^{\alpha}$ \\ 
        \cline{2-8}
        & $3c+2$ & $\mathbb{S} ^{\alpha}$ & $\mathbb{S} ^{\alpha}$ & \text{pt} & \text{pt} & $\mathbb{S} ^{\alpha}$ & $\mathbb{S} ^{\alpha}$ \\ 
        \hline
        \multirow[c]{3}{*}{$n=3a+1$} & $3c$ & $\mathbb{S} ^{\alpha-2}$ & $\mathbb{S} ^{\alpha-2}$ & \text{pt} & \text{pt} & \text{pt} & \text{pt} \\ 
        \cline{2-8}
        & $3c+1$ & \text{pt} & \text{pt} & \text{pt} & \text{pt} & \text{pt} & \text{pt} \\ 
        \cline{2-8}
        & $3c+2$ & $\mathbb{S} ^{\alpha-1}$ & $\mathbb{S} ^{\alpha-1}$ & \text{pt} & \text{pt} & \text{pt} & \text{pt} \\ 
        \hline
        \multirow[c]{3}{*}{$n=3a+2$} & $3c$ & \text{pt} & \text{pt} & \text{pt} & \text{pt} & \text{pt} & \text{pt} \\ 
        \cline{2-8}
        & $3c+1$ & \text{pt} & \text{pt} & \text{pt} & \text{pt} & $\mathbb{S} ^{\alpha+1}$ & $\mathbb{S} ^{\alpha+1}$ \\ 
        \cline{2-8}
        & $3c+2$ & \text{pt} & \text{pt} & \text{pt} & \text{pt} & $\mathbb{S} ^{\alpha+1}$ & $\mathbb{S} ^{\alpha+1}$ \\ 
        \hline
    \end{tabular}
    \caption{Computation for $\Dl$  and $\Lk$ }
    \label{tab:label1}
\end{table}

\begin{table}[H]
    \centering
    \begin{tabular}{|c|c|c|c|c|}
        \hline
        & \backslashbox{$m-k$ }{$k$ } & $3b$ & $3b+1$ & $3b+2$\\ 
        \hline 
        \multirow[c]{3}{*}{$n=3a$} & $3c$ & $\mathbb{S} ^{\alpha-1}$ & $\mathbb{S} ^{\alpha-1}$ & \text{pt}\\ 
        \cline{2-5}
        & $3c+1$ & \text{pt} & $\mathbb{S} ^{\alpha+1}$  & $\mathbb{S} ^{\alpha+1}$ \\ 
        \cline{2-5}
        & $3c+2$ & $\mathbb{S} ^{\alpha}$ & $\mathbb{S} ^{\alpha}\vee \mathbb{S} ^{\alpha+1} $ & $\mathbb{S} ^{\alpha+1}$ \\ 
        \hline
        \multirow[c]{3}{*}{$n=3a+1$} & $3c$ & $\mathbb{S} ^{\alpha-2}$ & $\mathbb{S} ^{\alpha-2}$ & \text{pt}\\ 
        \cline{2-5}
        & $3c+1$ & \text{pt} & \text{pt} & \text{pt} \\ 
        \cline{2-5}
        & $3c+2$ & $\mathbb{S} ^{\alpha-1}$ & $\mathbb{S} ^{\alpha-1}$ & \text{pt}\\ 
        \hline
        \multirow[c]{3}{*}{$n=3a+2$} & $3c$ & \text{pt} & \text{pt} & \text{pt}\\ 
        \cline{2-5}
        & $3c+1$ & \text{pt} & $\mathbb{S} ^{\alpha+2}$  & $\mathbb{S} ^{\alpha+2}$ \\ 
        \cline{2-5}
        & $3c+2$ & \text{pt} & $\mathbb{S} ^{\alpha+2}$  & $\mathbb{S} ^{\alpha+2}$ \\ 
        \hline
    \end{tabular}
    \caption{Independence complex}
    \label{tab:label2}
\end{table}
\end{proof}

Finally, we will compute the independence complex of a wedge of $k$ many path graphs. Note that the wedge graphs vary with the choice of wedge point. Furthermore, the graph obtained after choosing one wedge point, it can be viewed as \emph{terminal wedge of path graphs}.

\begin{figure}[!htb]
    \centering
    \def\svgwidth{0.5\columnwidth}
    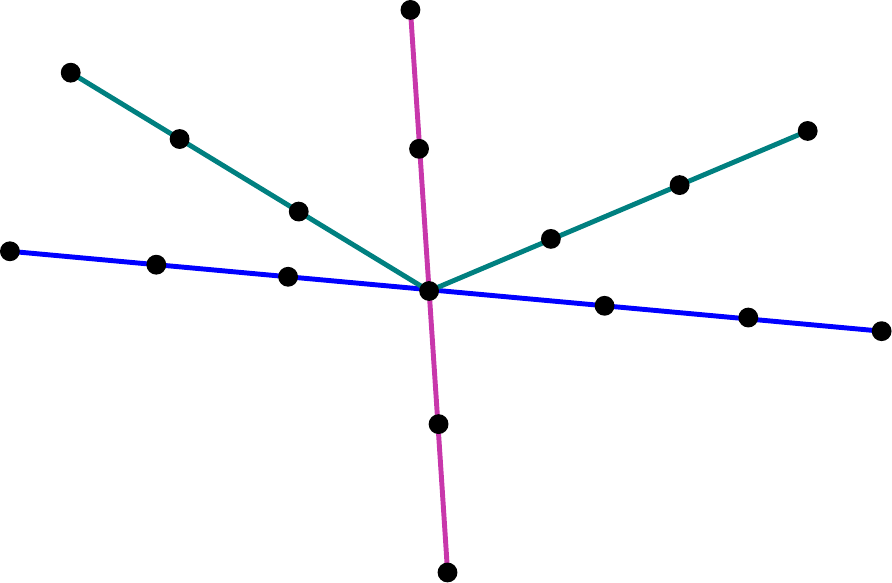

    \caption{$P_7\vee P_7 \vee P_5 = P_4\vee P_4\vee P_4\vee P_4\vee P_3\vee P_3$\label{fig:PathWedgePathMany}}
\end{figure}

\begin{theorem}\label{thm:pathWedgePathMany}
    The independence complex $\I \left(\bigvee_{i=1}^k P_{m_i}\right)$ of terminal wedge $k$ many path graphs is either a point or a sphere.
\end{theorem}

\begin{proof}
    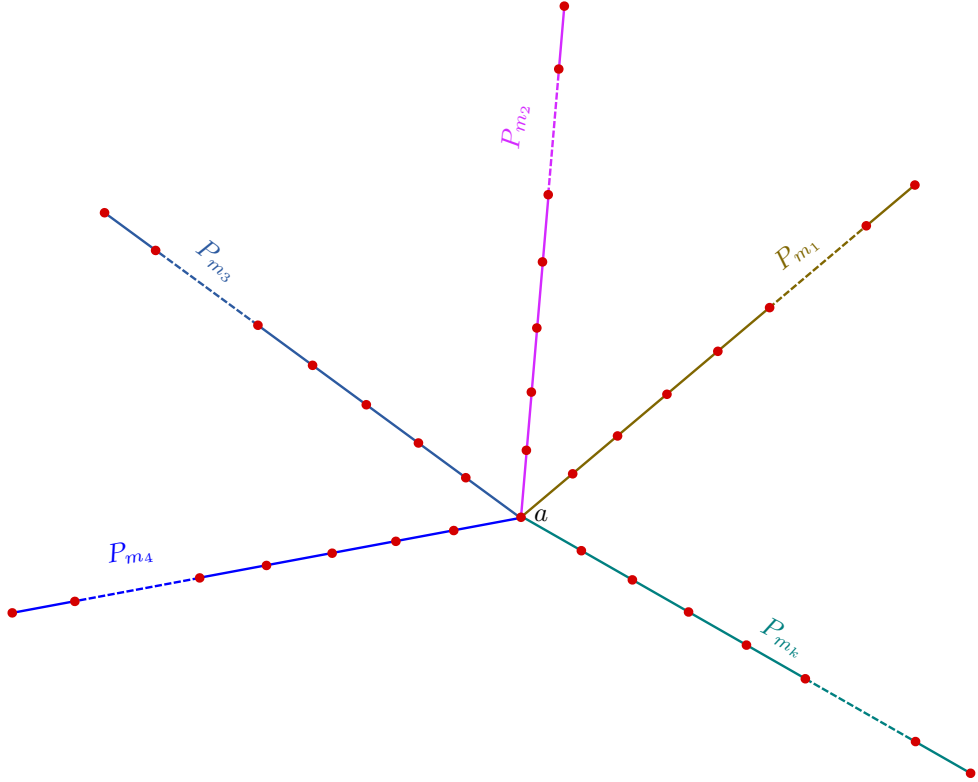
\begin{figure}[!htb]
        \centering
    \def\svgwidth{0.8\columnwidth}
\begingroup%
  \makeatletter%
  \providecommand\color[2][]{%
    \errmessage{(Inkscape) Color is used for the text in Inkscape, but the package 'color.sty' is not loaded}%
    \renewcommand\color[2][]{}%
  }%
  \providecommand\transparent[1]{%
    \errmessage{(Inkscape) Transparency is used (non-zero) for the text in Inkscape, but the package 'transparent.sty' is not loaded}%
    \renewcommand\transparent[1]{}%
  }%
  \providecommand\rotatebox[2]{#2}%
  \newcommand*\fsize{\dimexpr\f@size pt\relax}%
  \newcommand*\lineheight[1]{\fontsize{\fsize}{#1\fsize}\selectfont}%
  \ifx\svgwidth\undefined%
    \setlength{\unitlength}{856.03989488bp}%
    \ifx\svgscale\undefined%
      \relax%
    \else%
      \setlength{\unitlength}{\unitlength * \real{\svgscale}}%
    \fi%
  \else%
    \setlength{\unitlength}{\svgwidth}%
  \fi%
  \global\let\svgwidth\undefined%
  \global\let\svgscale\undefined%
  \makeatother%
  \begin{picture}(1,0.80855731)%
    \lineheight{1}%
    \setlength\tabcolsep{0pt}%
    \put(0,0){\includegraphics[width=\unitlength,page=1]{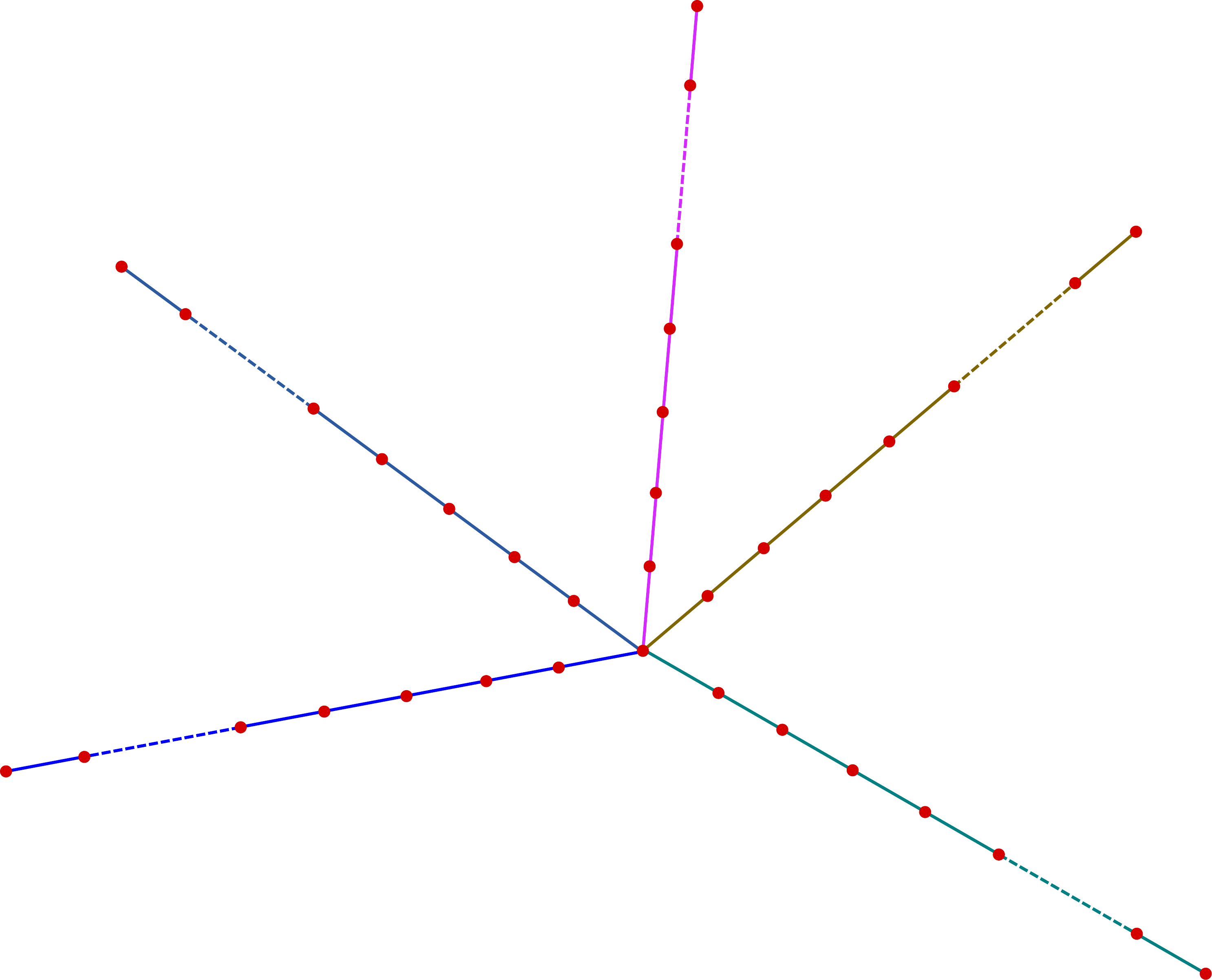}}%
    \put(0.54417014,0.26804129){\makebox(0,0)[lt]{\lineheight{1.25}\smash{\begin{tabular}[t]{l}$a$\end{tabular}}}}%
    \put(0.80054771,0.52751963){\color[rgb]{0.50196078,0.4,0}\rotatebox{40.042593}{\makebox(0,0)[lt]{\lineheight{1.25}\smash{\begin{tabular}[t]{l}$P_{m_1}$\end{tabular}}}}}%
    \put(0.52818983,0.6540887){\color[rgb]{0.83137255,0.16470588,1}\rotatebox{83.93477}{\makebox(0,0)[lt]{\lineheight{1.25}\smash{\begin{tabular}[t]{l}$P_{m_2}$\end{tabular}}}}}%
    \put(0.19334054,0.5483635){\color[rgb]{0.17254902,0.35294118,0.62745098}\rotatebox{-36.840006}{\makebox(0,0)[lt]{\lineheight{1.25}\smash{\begin{tabular}[t]{l}$P_{m_3}$\end{tabular}}}}}%
    \put(0.10396932,0.22376125){\color[rgb]{0,0,1}\rotatebox{10.424058}{\makebox(0,0)[lt]{\lineheight{1.25}\smash{\begin{tabular}[t]{l}$P_{m_4}$\end{tabular}}}}}%
    \put(0.77672246,0.15269643){\color[rgb]{0,0.50196078,0.50196078}\rotatebox{-27.900282}{\makebox(0,0)[lt]{\lineheight{1.25}\smash{\begin{tabular}[t]{l}$P_{m_k}$\end{tabular}}}}}%
  \end{picture}%
\endgroup%

        \caption{Terminal wedge of path graphs\label{fig:sample}}
    \end{figure}
    First we focus on the case if one of them is of the form $3m$. Without loss of generality we will assume that $m_1=3m$. 
    \begin{figure}[h]
        \centering
        \includegraphics[scale=.9]{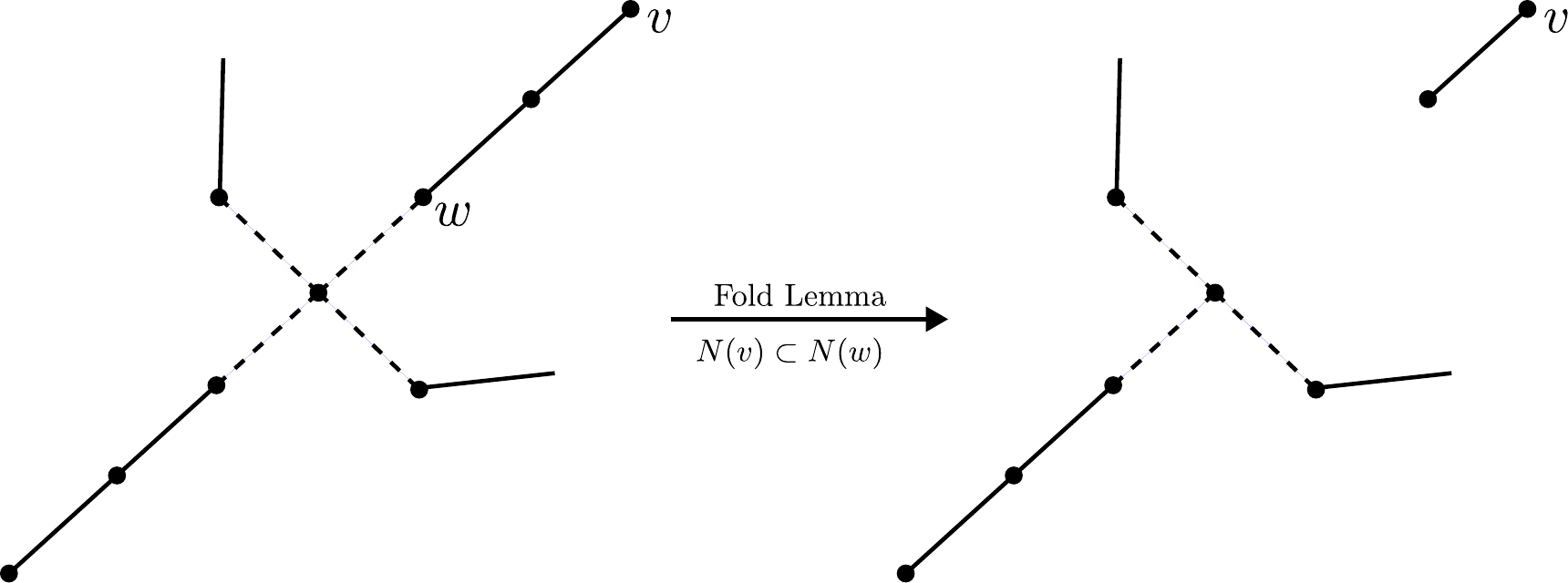}
        \caption{Reduction step}
        \label{fig:reduction-step}
    \end{figure}
    Applying \autoref{lemma:foldLemma} (see \Cref{fig:reduction-step}) to $P_{m_1}$, we get that 
    \begin{equation}\label{eq:PathWedge3m}
        \I \left(P_{m_1}\vee P_{m_2}\vee\cdots\vee P_{m_k}\right) = \left(\I(P_2)\right)^{\star(m)} \star \I(P_{m_2-1}) \star \cdots \star \I(P_{m_k-1}).
    \end{equation} 
    Next, we will compute the independence complex of a wedge of path graphs when all of them are either of the form $3l_i+1$ or $3l_i+2$. Now we consider the case if all of them are of the form $3l_i+1$. Then using the \autoref{lemma:foldLemma} (See \cref{fig:reduction-step}), we get that independence complexes 
    \begin{align*}
         \I\left(\bigvee_{i=1}^k P_{m_i}\right) & = \I(\{a\}) \star \I(P_2)^{\star (l_1)} \star \I(P_2)^{\star(l_2)} \star \cdots \star \I(P_2)^{\star(l_k)} \\
         & = \textup{pt}.
    \end{align*}   
    Next, assume that all $m_i$'s are of the form $3l_i+2$. In this case, using \autoref{lemma:foldLemma}  (See \cref{fig:reduction-step}), we get that independence complexes 
    \begin{align*}
        \I\left(\bigvee_{i=1}^k P_{m_i}\right) & = \I(P_2) \star \I(P_2)^{\star (l_1)} \star \I(P_2)^{\star(l_2)} \star \cdots \star \I(P_2)^{\star(l_k)} \\
        & = \mathbb{S}^{l_1 + \cdots + l_k}.
   \end{align*} 
   Finally, we will compute if some of them are of the form $3l_i+1$ and some are of the form $3l_i+2$. We assume that $k=t+r$, where $m_i$'s are of the form $3l_i+1$ for $i=1,2,\cdots,t$ and 
   the rest of them are of the form $3l_i+2$ for $i=t+1,\cdots,t+r$. Applying \autoref{lemma:foldLemma} repeatedly (See \cref{fig:reduction-step}), we obtain 
   \begin{align*}
        I \left(\bigvee_{i=1}^k P_{m_i}\right) & = \mathbb{S}^{l_1+\cdots+l_k-1}.
   \end{align*}
\end{proof}

\bibliographystyle{siam}

\end{document}